\documentclass[reqno]{amsart}

\numberwithin{equation}{section} 
\usepackage{amsmath,amsfonts,amssymb,amsthm,latexsym}
\usepackage{enumerate}
\usepackage{cancel}
\usepackage{mathrsfs}
\usepackage{stackrel}

\usepackage{cases}
\usepackage[square,comma,numbers,sort&compress]{natbib}

\usepackage[colorlinks=true]{hyperref}
\hypersetup{urlcolor=blue, citecolor=red}

\newcounter{mnote}
\theoremstyle{plain}
\newtheorem{theorem}{Theorem}[section]
\newtheorem{proposition}[theorem]{Proposition}
\newtheorem{lemma}[theorem]{Lemma}

\theoremstyle{definition}

\theoremstyle{remark}
\newtheorem{remark}[theorem]{Remark}

\newcommand{\field}[1]{\mathbb{#1}}
\newcommand{\nN}{\field{N}}

\newcommand{\nR}{\field{R}}

\newcommand{\cP}{\mathcal P}

\newcommand{\vphi}{\varphi}

\newcommand{\sand}{\quad\text{and}\quad}

\newcommand{\pd}[2]{\frac{\partial #1}{\partial #2}}
\newcommand{\od}[2]{\frac{d #1}{d #2}}

\newcommand{\abs}[1]{\left\lvert#1\right\rvert}
\newcommand{\norm}[1]{\left\lVert#1\right\rVert}
\newcommand{\set}[1]{\left\{#1\right\}}

\newcommand{\LpP}[1]{\text{$L^{#1}$}}
\newcommand{\HpP}[1]{\text{$H^{#1}$}}

\newcommand*{\magenta}[1]{{\color{black}{#1}}}

\begin{document}
\title[Data assimilation for the 2D B\'enard convection]{Continuous data assimilation for the 2D B\'enard convection through velocity measurements alone}

\date{\today}

%
\author{Aseel Farhat}
\address[Aseel Farhat]{Department of Mathematics\\
                Indiana University, Bloomington\\
        Bloomington , IN 47405, USA}
\email[Aseel Farhat]{afarhat@indiana.edu}
\author{Michael S. Jolly}
\address[Michael S. Jolly]{Department of Mathematics\\
                Indiana University, Bloomington\\
        Bloomington , IN 47405, USA}
\email[Michael S. Jolly]{msjolly@indiana.edu}
\author{Edriss S. Titi}
\address[Edriss S. Titi]{Department of Mathematics, Texas A\&M University, 3368 TAMU,
 College Station, TX 77843-3368, USA.  {\bf ALSO},
  Department of Computer Science and Applied Mathematics, Weizmann Institute
  of Science, Rehovot 76100, Israel.} \email{titi@math.tamu.edu and
  edriss.titi@weizmann.ac.il}


{\center To appear in Physica D: Nonlinear Phenomena\\ } 

\begin{abstract}
An algorithm for continuous data assimilation for the two-dimensional B\'enard convection problem is introduced and analyzed. It is inspired by the data assimilation algorithm developed for the Navier-Stokes equations, which allows for the implementation of variety of observables: low Fourier modes, nodal values, finite volume averages, and finite elements. The novelty here is that the observed data is obtained for the velocity field alone; i.e. no temperature measurements are needed for this algorithm. We provide conditions on the spatial resolution of the observed data, under the assumption that the observed data is free of noise, which are sufficient to show that the solution of the algorithm approaches, at an exponential rate, the unique exact unknown solution of the B\'enard convection problem associated with the observed (finite dimensional projection of) velocity.
\end{abstract}

 \maketitle
 {\bf MSC Subject Classifications:} 35Q30, 93C20, 37C50, 76B75, 34D06.

{\bf Keywords:} Continuous data assimilation, two-dimensional B\'enard convection problem, determining projections.\\
\section{Introduction}\label{intro}

Accurate numerical simulations of nonlinear systems require high precision in the initial data.
For most applications however, initial data which
should ideally be defined on the whole physical domain, can be measured only discretely, often with inadequate resolution. Data assimilation refers to the process of completing, or enhancing the resolution of the initial condition.
The classical method of continuous data assimilation, see, e.g., \cite{Daley}, is
to insert observational measurements directly into a model as the
latter is being integrated in time.
The natural mathematical target for data assimilation is the global attractor.
This set contains all the long time behavior; it is compact, invariant, and finite-dimensional. Another key notion
is that of determining parameters. A projection (onto say a finite number of low Fourier modes, or other types of interpolant projections based on nodal values and volume elements) is said to be determining if, whenever the projection of two trajectories on the global attractor approach each other, as $t \to \infty$, the full trajectories approach each other, see, for example,
\cite{C_O_T, F_M_R_T, Foias_Prodi, Foias_Temam, Foias_Temam_2, Foias_Titi, Holst_Titi, Jones_Titi, Jones_Titi_2} and references therein.  One way to exploit this is to insert low mode observables from a time series into the equation for the evolution of the high modes.  After a relatively short time interval $[t_{-1},t_0]$ the solution to the equation for the high modes is close to the high modes of the exact solution associated with the observables.
At that point the low modes and high modes can be combined to form a complete good approximation of the state of the system at time $t = t_0$, which can then be used as an initial condition for a high resolution simulation. This was the approach taken for the 2D Navier-Stokes in \cite{Browning_H_K, B_L_Stuart, Henshaw_Kreiss_Ystrom,Hayden_Olson_Titi,Olson_Titi_2003, Olson_Titi_2008, Korn}. Except \magenta{for} the work in \cite{B_L_Stuart} for the 3DVAR Gaussian filter, and \cite{Bessaih_Olson_Titi} using the determining parameters nudging approach of this paper for data assimilation, the previously mentioned theoretical work assumed that the observational measurements are error free. Notably, the authors of \cite{Hayden_Olson_Titi} present an algorithm for data assimilation that uses discrete in space and time measurements.

An alternative approach in \cite{Azouani_Titi} uses the observables in a feedback control term.
The advantage is that, since no derivatives are required
of the coarse grain observable, this works for a general class of interpolant operators. The main idea
can be outlined in terms of a general evolutionary equation
\begin{align}\label{ev_eq}
\od{u}{t} = F(u),
\end{align}
where the initial data $u(0)= u_0$ is missing. Let $I_h(u(t))$ represent an interpolant operator based on the spatial observations of the reference solution of system \eqref{ev_eq} at a coarse spatial resolution of size $h$.
Consider
\begin{subequations}\label{azouani}
\begin{align}
&\od{v}{t} = F(v) - \mu (I_h(v)- I_h(u)), \\
&v(0)= v_0,
\end{align}
\end{subequations}
where $\mu>0$ is a relaxation (nudging) parameter, and $v_0$ is arbitrary.
It is shown in \cite{Azouani_Olson_Titi} that if one takes $\mu$ large enough, and $h$
small enough (depending on $\mu$), then  $v(t)$ converges to the reference solution $u(t)$ of the two-dimensional Navier-Stokes equations, at an exponential rate,
as $t \to \infty$. An extension to this approach of \cite{Azouani_Olson_Titi} to the case where the observations are contaminated with random errors is studied in \cite{Bessaih_Olson_Titi}. The feedback control approach to data assimilation
plays a key role in the derivation in \cite{FJKT2}
of \magenta{a} determining form for the 2D NSE, which is an
ordinary differential equation 
whose steady states are precisely the trajectories in the global attractor.

The B\'enard convection problem is a model of the convection of an incompressible fluid layer in a box $(0, L) \times(0, 1)$ which is heated from below in such a way that the lower plate is maintained at a temperature $T_0$ while the upper one is at temperature $T_1<T_0$, where $T_0$ and $T_1$ are constants. 
After \magenta{a} change of variables (see \cite{Foias_Manley_Temam}), the non-dimensional two-dimensional Boussinesq equations that govern the velocity of the fluid, the pressure $p$, and the normalized temperature (or the density) of the fluid read 
\begin{subequations}\label{Bous}
\begin{align}
&\pd{u}{t} - \nu\Delta u + (u\cdot\nabla)u + \nabla p = \theta \mathbf{e}_2, \label{Bous1}\\
&\pd{\theta}{t} - \kappa\Delta\theta + (u\cdot\nabla)\theta - u\cdot{\bf e}_2=0, \label{Bous2}\\
&\nabla\cdot u= 0,\label{Bous_div}\\
&u(0;x) = u_0(x), \quad \theta(0;x)=\theta_0(x),\label{Bous_initial} 
\end{align}
with the boundary conditions in the $\magenta{x_2}$-direction
\begin{align}
u, \theta=0 \quad \text{at} \quad \magenta{x_2}=0 \quad \text{and} \quad \magenta{x_2}=1, \label{boundary1}
\end{align}
and in the $\magenta{x_1}$-direction, for simplicity, we will impose a periodicity condition
\begin{align}
& u, \theta, p \text{ are periodic, of period } L, \text{ in the }\magenta{x_1}\text{-direction}.\label{boundary3} 
\end{align}
\end{subequations}

The Boussinesq system \eqref{Bous} is usually referred to as the B\'enard convection problem.  The global regularity of the two-dimensional Boussinesq equations was established in \cite{Cannon_DiB_1980} (see also\cite{Temam_1997}) following the classical methods for the Navier-Stokes equations. The mathematical analysis of system \eqref{Bous} has been studied in \cite{Foias_Manley_Temam} (see also \cite{Temam_1997}), where the existence and uniqueness of weak solution in dimension two and three were proved, along with the existence of a finite dimensional global attractor was also established in space dimension two.  It was also shown in \cite{Foias_Manley_Temam} that system \eqref{Bous} can be handled with different boundary conditions. 
For recent results concerning the two-dimensional Boussinesq equations we refer the reader to \cite{Chae_2006}, \cite{Danchin_Paicu}, \cite{Hmidi_Keraani}, \cite{Hou_Li_2005}, \magenta{\cite{Hu_Kukavica_Ziane}}, \cite{Larios_Lunasin_Titi_2010}, and references therein.

In this work, we present a new continuous data assimilation algorithm for the B\'enard convection problem \eqref{Bous}. 
The twist here is that we can recover the reference solution to \eqref{Bous} using coarse-grain
data for the velocity alone; {\it temperature data is not needed}.  This is done by solving
\begin{subequations}\label{DA_Bous}
\begin{align}
&\pd{v}{t} -\nu \Delta v + (v\cdot\nabla)v +\nabla \tilde p = \eta\mathbf{e}_2- \mu(I_h(v)-I_h(u)), \\
&\pd{\eta}{t} -\kappa\Delta\eta + (v\cdot\nabla)\eta - v\cdot {\bf e}_2= 0, \\
&\nabla \cdot v = 0, \\
&v(0;x) = v_0(x), \quad \eta(0; x) = \eta_0(x), \label{DA_Bous_initial}
\end{align}
with the boundary conditions
\begin{align}\label{boundary}
&\eqref{boundary1}, \text{ and }\eqref{boundary3}\text{ hold with }u, \theta, \text{ and } p\text{ replaced by }v, \eta, \text{ and }{\tilde p},\notag \\
&\text{respectively}.
\end{align}
\end{subequations}

Here, $\tilde p$ is a modified pressure, and as for \eqref{azouani}, $v_0, \eta_0$ may \magenta{be} chosen arbitrarily, e.g., zero
in each case.  If we knew $u_0$ and $\theta_0$ in \eqref{Bous_initial}, then we could take $v_0=u_0$ and $\eta_0=\theta_0$ in \eqref{DA_Bous_initial} and the solution $(v,\eta)$ would be identically $(u, \theta)$, by the uniqueness of solutions of system \eqref{DA_Bous}, which will be shown below. The point, again, is that in many applications, we do not
know $u_0$ and $\theta_0$.  We emphasize that in this algorithm, we construct our approximate solutions $(v,\eta)$ using only the observations of the velocity field solution, $I_h(u)$, in the $v$-equation; no observations $I_h(\theta)$ are needed for the temperature.

We will consider two types of interpolant observables.
One is to be given by a linear interpolant operator $I_h: \HpP{1} \rightarrow \LpP{2}$ satisfying
 the approximation property
\begin{align}\label{app}
\norm{\varphi - I_h(\varphi)}_{\LpP{2}} \leq \gamma_0h\norm{\varphi}_{\HpP{1}}, 
\end{align}
for every $\varphi \in \HpP{1}$, where $\gamma_0>0$ is a dimensionless constant.
The other type is given by $I_h: \HpP{2}\rightarrow\LpP{2}$, together with
\begin{align}\label{app2}
\norm{\varphi - I_h(\varphi)}_{\LpP{2}} \leq \gamma_1h\norm{\varphi}_{\HpP{1}} + \gamma_2h^2\norm{\varphi}_{\HpP{2}},
\end{align}
for every $\varphi \in \HpP{2}$, where $\gamma_1, \gamma_2>0$ are dimensionless constants.
One example of an interpolant observable that satisfies \eqref{app} is the orthogonal projection onto the low Fourier modes with wave numbers $k$ such that $|k|\leq 1/h$. A more physical example is the volume elements that were studied in \cite{Jones_Titi}. An example of an interpolant observable that satisfies \eqref{app2} is given by the measurements at a discrete set of nodal points in $\Omega$ (see Appendix A in \cite{Azouani_Olson_Titi}).
%


In the next section we lay out the functional setting commonly used in the mathematical study of the Navier-Stokes equations.  We also recall \magenta{the} previous work on the B\'enard problem establishing well-posedness and existence of a global attractor.  In section \ref{conv} we prove that solutions on the global attractor of \eqref{Bous}
are determined by the velocity alone, a fact which motivates our data assimilation algorithm \eqref{DA_Bous}.  The main results are in section \ref{conv}. Assuming adequate resolution in the observational data, and separately
conditions \eqref{app} and \eqref{app2}, we prove the well-posedness of system \eqref{DA_Bous}
as well as convergence (at an exponential rate) of the approximate solution $(v,\eta)$ of \eqref{DA_Bous}  to the reference solution $(u,\theta)$ \magenta{of} the B\'enard convection problem \eqref{Bous}.

\bigskip
\section{Preliminaries}\label{pre}

For the sake of completeness, this section presents some preliminary material and notation commonly used in the mathematical study of fluids, in particular in the study of the Navier-Stokes equations (NSE) and the Euler equations. For more detailed discussion on these topics, we refer the reader to \cite{Constantin_Foias_1988}, \cite{Robinson}, \cite{Temam_1995_Fun_Anal} and \cite{Temam_2001_Th_Num}.

We begin by defining function spaces corresponding to the relevant physical boundary conditions. We define $\mathcal{F}$ to be the set of $C^\infty(\Omega)$ functions defined in $\Omega$, which are trigonometric polynomials in $\magenta{x_1}$ with period $L$, and compactly supported in the $\magenta{x_2}$-direction. We denote the space of smooth vector-valued functions which incorporates the divergence-free condition by
\[\mathcal{V}:=\set{\phi\in\mathcal{F}\times\mathcal{F}: \; \nabla\cdot\phi=0}.\]

\begin{remark}
We will use the same notation indiscriminately for both scalar and vector Lebesgue and Sobolev spaces, which should not be a source of confusion.
\end{remark}

The closures of $\mathcal{V}$ and $\mathcal{F}$ in $L^2(\Omega)$ will be denoted by $H_0$ and $H_1$, respectively. $H_0$ and $H_1$ will be endowed  \magenta{with} the usual scalar product
\[(u,v)_{H_0}=\sum_{i=1}^2\int_{\Omega} u^i(x)v^i(x)\,dx
\sand
(\psi,\phi)_{H_1}=\int_{\Omega} \psi(x)\phi(x)\,dx, \]
and the associated norms $\norm{u}_{H_0} = (u,u)_{H_0}^{1/2}$ and $\norm{\phi}_{H_1} = (\phi,\phi)_{H_1}^{1/2}$, respectively. We denote by the closures of $\mathcal{V}$ and $\mathcal{F}$ in $H^1(\Omega)$ by $V_0$ and $V_1$, respectively. $V_0$ and $V_1$ are Hilbert spaces endowed by the scalar product
\[((u,v))_{V_0}=\sum_{i,j=1}^2\int_{\Omega}\partial_ju^i(x)\partial_jv^i(x)\,dx
\sand
((\psi,\phi))_{V_1}=\sum_{j=1}^2\int_{\Omega}\partial_j\psi(x)\partial_j\phi(x)\,dx, \]
and the associated norms $\norm{u}_{V_0} = ((u,u))_{V_0}^{1/2}$ and $\norm{\phi}_{V_1} = ((\phi,\phi))_{V_1}^{1/2}$, respectively.

Let $D(A_0)= V_0\cap H^2(\Omega)$ and $D(A_1)= V_1\cap H^2(\Omega)$ and let $A_i: D(A_i) \rightarrow H_i$ be the unbounded linear operator defined by
\[ (A_i u, v)_{H_i} = ((u,v))_{V_i}, \qquad i = 0,1, \]
for all $u, v \in D(A_i)$. The operator $A_i$ is self-adjoint and $A_i^{-1}$ is a compact , positive-definite, self-adjoint linear operator in $H_i$, for each $i=0,1$. Thus, there exists a complete orthonormal set of eigenfunctions $w_j^i$ in $H_i$ such that $A_iw_j^i= \lambda_j^iw_j^i$ where $0<\lambda_j^i\leq\lambda_{j+1}^i$ for $j\in \nN$ and each $i=0,1$.


We denote the Helmholtz-Leray projector from $L^2(\Omega)$ onto $H_0$ by $\cP_\sigma$ and the dual of $V_i$ by $V_i^{'}$, for $i=0,1$. We define a map $B_0:V_0\times V_0 \rightarrow V_0^{'}$ by
$$\langle B_0(u,v), w\rangle_{V_0, V_0^{'}} = (((u\cdot\nabla)v),w)_{H_0}, $$
for each $u, v,w \in V_0$,
and its scalar analogue
$B_1:V_0\times V_1 \rightarrow V_1^{'}$ by
$$\langle B_1(u,v), w\rangle_{V_1, V_1^{'}} = (((u\cdot\nabla)v),w)_{H_1}, $$
for each $u\in V_0$ and $v,w \in V_1$.
These bilinear operators have the algebraic property
\begin{subequations}\label{prop1}
\begin{align}
\langle B_0(u,v), w\rangle_{V_0,V_0^{'}} = - \langle B_0(u,w),v\rangle_{V_0,V_0^{'}},\end{align}
and
\begin{align}
\langle B_1(u,\theta),\phi\rangle_{V_1,V_1^{'}} = - \langle B_1(u,\phi),\theta\rangle_{V_1,V_1^{'}},
\end{align}
\end{subequations}
for each $u \in V_0$ and $v, w \in V_0$ and $\theta, \phi \in V_1$. Consequently, they also enjoy
the orthogonal property
\begin{align}\label{orth}
\langle \magenta{B_0}(u,v), v\rangle_{V_0,V_0^{'}} = 0, \qquad \text{and} \qquad \langle B_1(u,\theta), \theta\rangle_{V_1,V_1^{'}} = 0,
\end{align}
for each $u\in V_0$, $v\in V_0$ and $\theta\in V_1$.

In the above notation, we write the incompressible two-dimensional B\'enard convection problem \eqref{Bous} in the functional form
\begin{subequations}\label{Bous_fun}
\begin{align}
&\od{u}{t} + \nu A_0u + B_0(u,u) = \cP_\sigma (\theta {\bf e}_2), \label{Bous_fun_1} \\
& \od{\theta}{t} + \kappa A_1 \theta + B_1(u,\theta) - u \cdot{\bf e}_2= 0, \label{Bous_fun_2}\\
&u(0;x) = u_0(x), \quad \theta(0;x) = \theta_0(x).
\end{align}
\end{subequations}


Next, we recall Ladyzhenskaya's inequality for an integrable function $\vphi \in V_i$, $i=0, 1$:
\begin{equation}\label{L4_to_H1}
\|\vphi\|_{\LpP{4}}^2\leq c_1 \|\vphi\|_{\LpP{2}}\norm{\vphi}_{V_i},
\end{equation}
where $c_1$ is a universal, dimensionless, positive constant.
Hereafter, $c$ denotes a generic constant which may change from line to line. We also have the Poincar\'e inequality:
\begin{subequations}\label{poincare}
\begin{align}
\|\vphi\|_{\LpP{2}}^2\leq \lambda_1^{-1}\|\vphi\|_{V_i}^2, &\quad \text{ for all } \varphi\in {V_i}, \\
\|\vphi\|_{V_i}^2\leq \lambda_1^{-1}\|A_i\vphi\|_{\LpP{2}}^2, &\quad \text{ for all } \varphi\in \mathcal{D}(A_i),
\end{align}
\end{subequations}
where $\lambda_1$ is the minimum of the two smallest eigenvalues of the Stokes operators $A_i$, $i=0,1$.

Furthermore, inequalities \eqref{app} and \eqref{app2} imply that
\begin{align}\label{app_F}
\norm{\cP_\sigma(w-I_h(w))}_{H_0}^2\leq c_0^2h^2\norm{w}_{{V_0}}^2,
\end{align}
for every $w\in V_0$, where $c_0=\gamma_0$, and respectively,
\begin{align}\label{app2_F}
\norm{\cP_\sigma(w-I_h(w))}_{{H_0}}^2\leq \frac{1}{2}c_0^2h^2\norm{w}_{{V_0}}^2+ \frac14c_0^4h^4\norm{A_0w}_{{H_0}}^2,
\end{align}
for every $w\in \mathcal{D}(A_0)$, for some $c_0>0$ that depends only on $\gamma_1$ and $\gamma_2$.

We will apply the following inequality which is a particular case of a more general inequality proved in \cite{Jones_Titi}.
\begin{lemma}\label{gen_gron_2}\cite{Jones_Titi} Let $\tau>0$ be fixed. Suppose that $Y(t)$ is an absolutely continuous function which is locally integrable and that it satisfies the following:
\begin{align*}
\od{Y}{t} + \alpha(t)Y \leq 0,\qquad \text{ a.e. on } (0,\infty),
\end{align*}
and
\begin{align*}
\liminf_{t\rightarrow\infty} \int_t^{t+\tau} \alpha(s)\,ds \geq \gamma, \qquad 
\limsup_{t\rightarrow\infty} \int_t^{t+\tau}  \alpha^{-}(s)\,ds < \infty,
\end{align*}
for some $\gamma>0$, where $\alpha^{-} = \max\{\magenta{-\alpha}, 0\}$.
Then, $Y(t)\rightarrow 0$ at an exponential rate, as $t\rightarrow \infty$.
\end{lemma}

We recall the following existence and uniqueness results from \cite{Foias_Manley_Temam, Temam_1997} for the B\'enard convection problem \eqref{Bous_fun}.

\begin{theorem}[Existence and Uniqueness of Weak Solutions]
Let $T>0$ be fixed, but arbitrary. Let $\nu>0$ and $\kappa>0$. If $u_0\in H_0$ and $\theta_0\in {H_1}$, then system \eqref{Bous_fun} has a unique weak solution $(u, \theta)$ such that $u\in C([0,T];H_0) \cap L^2([0,T];V_0)$ and $\theta\in C([0,T];{H_1})$.
\end{theorem}

It was also shown in \cite{Foias_Manley_Temam, Temam_1997} that  the 2D B\'enard convection system has a finite-dimensional global attractor.

\begin{theorem}[Existence of a Global Attractor]\label{global_attractor_Bous} Let $T>0$ be fixed, but arbitrary. If the initial data $u_0\in V_0$ and $\theta_0\in{V_1}$, then the system \eqref{Bous_fun} has a unique strong solution $(u,\theta)$ that satisfies $u\in C([0,T];V_0)\cap L^2([0,T];\mathcal{D}(A_0))$ and $\theta\in C([0,T];{V_1})\cap L^2([0,T];\mathcal{D}(A_1))$. Moreover, the system induced by \eqref{Bous_fun} is well-posed and possesses a finite-dimensional global attractor $\mathcal{A}$ which is maximal among all the bounded invariant sets and is compact in $H_0\times {H_1}$.
\end{theorem}

We will use the following bounds on $(u,\theta)$ later in our analysis. 
%
%
\begin{proposition}[Uniform Bounds on the solutions]\label{unif_bounds}\cite{Foias_Manley_Temam, Temam_1997}
Let $(u,\theta)$ be a strong solution of \eqref{Bous_fun}. There exists $t_0>0$, which depends on norms of the initial data, such that for $t\geq t_0$,
\begin{align}
\norm{\theta(t)}_{{H_1}} &\leq 2L^{1/2}, \quad \text{and}\quad
\norm{u(t)}_{{H_0}}  \leq \frac{2L^{1/2}}{\nu\lambda_1^{1/2}}, 
\end{align}
\begin{align}
&\int_t^{t+1}\norm{u(s)}_{{V_0}}^2\,ds \leq a_3, \quad
\int_t^{t+1}\norm{\theta(s)}_{{V_1}}^2\,ds \leq b_3,
\end{align}
\begin{align}
\norm{u(t)}_{{V_0}}^2 &\leq \left(a_2+a_3\right) e^{a_1} =: J_0, \label{J_0_eps}\\
\norm{\theta(t)}_{{V_1}}^2 & \leq (b_2+ b_3)e^{b_1} = : J_1, \label{J_1_eps}
\end{align}
where 
\begin{align*}
a_2 = b_2 = \frac{cL}{\nu\lambda_1^{1/2}}, \quad
a_3 = \frac{cL(1+\lambda_1^{1/2})}{\nu^2\lambda_1}, \quad b_3= \frac{cL(1+\nu\lambda_1^{1/2})}{\kappa\nu\lambda_1^{1/2}}, 
\end{align*}
and
\begin{align*}
a_1 = \frac{cL}{\nu^5\lambda_1}a_3, \quad b_1 =\frac{cL}{\kappa^3\nu^2\lambda_1}a_3, 
\end{align*}
for some dimensionless positive constant $c$.
\end{proposition}

\begin{remark}
Recall that all quantities in \eqref{Bous_fun} are dimensionless, including the parameters $\nu$, $\kappa$, $L$ and $\lambda_1$. The proof of Proposition \ref{unif_bounds} in \cite{Foias_Manley_Temam} was done \magenta{for} the particular case \magenta{$\lambda_1=1$}. We state it for arbitrary $\lambda_1$, so that our ultimate results can show this dependence. 
\end{remark}

\bigskip
\section{Convergence Results}\label{conv}
In this section, we derive conditions under which the approximate solution $(v,\eta)$ of the data assimilation system \eqref{DA_Bous_fun} converges to the solution $(u,\theta)$ of the B\'enard convection problem \eqref{Bous_fun} as $t\rightarrow \infty$. We will prove the result for observables operators that satisfy \eqref{app_F} and \eqref{app2_F}, in functional settings, respectively. 

The idea to apply data assimilation using observations of velocity only is inspired by the
fact that solutions in the global attractor of \eqref{Bous_fun} are completely determined by
their velocity component. That is, the values of the temperature (or the density)
$\theta(t;x)$ in $\mathcal{A}$ are completely determined by the velocity vector field
$u(t;x)$ for all time in $\mathcal{A}$.

\begin{proposition}\label{u_to_theta}
Let $(u_1(t;x), \theta_1(t;x))$ and $(u_2(t;x), \theta_2(t;x))$ be two trajectories in $\mathcal{A}$ of \eqref{Bous_fun}. If $u_1(t;x)=u_2(t;x)=u(t;x)$ in $\mathcal{A}$ for all $t\in \nR$, then $\theta_1(t;x)=\theta_2(t;x)$, for all $t\in \nR$.

\begin{proof}
Let $\theta_1$ and $\theta_2$ be two trajectories such that $(u(t;x), \theta_i(t;x))$ $\in$ $\mathcal{A}$ for $i=1, 2$ for all $t\in\nR$. Define ${\tilde\theta}= \theta_1-\theta_2$. Then, $\tilde\theta$ satisfies the equation
\begin{align}\label{t_theta}
\pd{\tilde\theta}{t} + \kappa A_1\tilde\theta +\magenta{B_1}(u, \tilde{\theta}) =0.
\end{align}
Taking the ${H_1}$ inner product of \eqref{t_theta} with $\tilde\theta$ and using the Poincar\'e inequality \eqref{poincare} yields
\begin{align*}
\frac12 \od{}{t}\norm{\tilde\theta}_{{H_1}}^2 &= -\kappa\norm{\tilde\theta}_{{V_1}}^2 \leq-\kappa \lambda_1\norm{\tilde\theta}_{{H_1}}^2.
\end{align*}
Using Gronwall's lemma we have
\begin{align}\label{t_theta_norm}
\norm{\tilde\theta(t)}_{{H_1}}^2 \leq e^{-\magenta{2}\kappa\lambda_1(t-s)} \norm{\tilde\theta(s)}_{{H_1}}^2,
\end{align}
for all $-\infty<s\leq  t <\infty.$
Since the solutions on the global attractor $\mathcal{A}$ are bounded in $H\times{H_1}$ and $V\times{V_1}$, we can let $s\rightarrow -\infty$ in \eqref{t_theta_norm} to obtain
\begin{align}
\norm{\tilde\theta(t)}_{{H_1}}^2 =0,
\end{align}
for all $t\in \nR$.
\end{proof}
\end{proposition}



In functional form the system \eqref{DA_Bous} reads as
\begin{subequations}\label{DA_Bous_fun}
\begin{align}
&\od{v}{t} + \nu A_0v + \magenta{B_0}(v,v) = \cP_\sigma(\eta {\bf e}_2) - \mu \cP_\sigma(I_h(v)-I_h(u)), \label{DA_v}\\
& \od{\eta}{t} +\kappa A_1\eta+ \magenta{B_1}(v,\eta) - v\cdot{\bf e}_2= 0, \label{DA_e}\\
&v(0;x)= v_0(x), \quad \eta(0;x)= \eta_0(x),
\end{align}
\end{subequations}
where $(u,\theta)$ is the strong solution of the 2D B\'enard convection problem \eqref{Bous_fun} on the global attractor $\mathcal{A}$.

Following the techniques that were introduced to prove the existence and uniqueness of solutions for the Navier-Stokes equations and the Boussinesq equations (see for example, \cite{Cannon_DiB_1980}, \cite{Constantin_Foias_1988}, \cite{Foias_Manley_Temam}, \cite{Temam_1997} and \cite{Temam_2001_Th_Num}). We can show the existence of the solution $(v,\eta)$ of system \eqref{DA_Bous_fun} using the Galerkin method and the Aubin compactness theorem. The uniqueness and the well-posedness will follow as in the case of the two-dimensional Navier-Stokes equations and the two-dimensional Boussinesq equations using the Lions-Magenes lemma and Gronwall's lemma.

First, we will prove that under certain conditions on $\mu$, the approximate solution $(v,\eta)$ of the data assimilation system \eqref{DA_Bous_fun} converges to the solution $(u,\theta)$ of the B\'enard problem \eqref{Bous_fun} as $t\rightarrow \infty$ when the observables operators satisfy \eqref{app_F}.

\begin{theorem}\label{th_conv_1}
Let $I_h$ satisfy the approximation property \eqref{app_F} and $(u(t;x),\theta(t;x))$ be a strong solution in the global attractor of \eqref{Bous_fun}. Let $\mu>0$ be arbitrary and $h<<1$ be chosen such that $\mu c_0^2h^2\leq \nu$, then \eqref{DA_Bous_fun} has a unique {\it strong} solution $(v,\eta)$ that satisfies
\begin{subequations}\label{strong}
\begin{align}
v \in C([0,T];V_0)\cap L^2([0,T];\mathcal{D}(A_0)),\\
\eta\in C([0,T];{V_1})\cap L^2([0,T];\mathcal{D}(A_1)),
\end{align}
and
\begin{align}
\od{v}{t} \in L^2([0,T];H_0), \qquad \od{\eta}{t}\in L^2([0,T];{H_1}).
\end{align}
\end{subequations}
Moreover, the strong solution $(v,\eta)$ depends continuously on the initial data in the $V_0\times{V_1}$ norm.

If we choose $\mu>0$ large enough such that
\begin{align}\label{mu_1}
\mu\geq \frac{8}{\kappa\lambda_1} + \frac{8c_1^2a_3}{\nu}+ \frac{8c_1^4J_1b_3}{\kappa^2\lambda_1\nu},
\end{align}
and $h>0$ small enough such that $\mu c_0^2h^2\leq \nu$, where the positive constants $a_3(\nu,L)$, $b_3(\nu,\kappa,L)$, and $J_1(\nu,\kappa,L)$, are defined in Proposition \ref{unif_bounds}. Then, $\norm{u(t)-v(t)}_{{H_0}}^2 + \norm{\theta(t)-\eta(t)}_{{H_1}}^2$ $\rightarrow 0$ at an
exponential rate as $t \rightarrow \infty$.
\end{theorem}

\begin{proof}
The existence of the solution $(v,\eta)$ follows using the Galerkin method and the Aubin compactness theorem. We refer the reader to the detailed proof of well-posedness of the data assimilation system in \cite{Azouani_Olson_Titi}. The argument here \magenta{would be} identical and will be omitted.

Define $w = u-v$, $\xi = \theta-\eta$. Then $w$ and $\xi$ satisfy the equations
\begin{subequations}
\begin{align}
&\od{w}{t} +\nu A_0w +\magenta{B_0}(v,w)+ \magenta{B_0}(w,u) = \cP_\sigma(\xi \mathbf{e}_2)- \mu \cP_\sigma I_h(w), \label{w}\\
& \od{\xi}{t} - \kappa A_1 \xi +\magenta{B_1}(v,\xi) + \magenta{B_1}(w,\theta)- w\cdot{\bf e}_2= 0, \label{xi}\\
&w(0;x) = w_0(x): = u_0(x)-v_0(x), \\
&\xi(0;x) = \xi_0(x) := \theta_0(x) - \eta_0(x).
\end{align}
\end{subequations}
Since $\od{w}{t}$ and $\od{\xi}{t}$ are bounded in $L^2([0,T];H_0)$ and $L^2([0,T];{H_1})$, respectively, using the Lions-Magenes lemma, we can take the ${\LpP{2}}$ inner product of \eqref{w} and \eqref{xi} with $w$ and $\xi$, respectively, and obtain
\begin{align*}
&\frac 12 \od{}{t} \norm{w}_{{H_0}}^2 + \nu \norm{w}_{V_0}^2 + \left(\magenta{B_0}(w,u),w\right) = \int_{\Omega} \xi (w\cdot{\bf e}_2)\, dx - \mu (I_h(w), w), \\
&\frac1 2 \od{}{t} \norm{\xi}_{{H_1}}^2 + \kappa\norm{\xi}_{{V_1}}^2 + \left(\magenta{B_1}(w,\theta),\xi\right) = \int_{\Omega} \xi (w\cdot{\bf e}_2)\, dx.
\end{align*}
By H\"older's inequality, Young's inequality and the Poincar\'e inequality \eqref{poincare}, we have
\begin{align}\label{1}
\abs{\int_{\Omega}\xi (w\cdot{\bf e}_2)\, dx} &\leq \norm{w}_{{H_0}}\norm{\xi}_{{H_1}}\notag \\
& \leq \frac{\kappa\lambda_1}{4}\norm{\xi}_{{H_1}}^2 + \frac{1}{\kappa\lambda_1}\norm{w}_{{H_0}}^2\notag \\
&\leq \frac{\kappa}{4}\norm{\nabla\xi}_{{H_1}}^2 +  \frac{1}{\kappa\lambda_1}\norm{w}_{{H_0}}^2.
\end{align}
Ladyzhenskaya's inequality \eqref{L4_to_H1} and Young's inequality yield
\begin{align}\label{2}
& \abs{\left(\magenta{B_1}(w,\theta), \xi\right)} \leq \norm{w}_{\LpP{4}}\norm{\xi}_{\LpP{4}}\norm{\theta}_{{V_1}}\notag \\
&\quad \leq c_1 \norm{w}_{{H_0}}^{1/2}\norm{w}_{{V_0}}^{1/2}\norm{\xi}_{{H_1}}^{1/2}\norm{\xi}_{{V_1}}^{1/2} \norm{\theta}_{{V_1}}\notag \\
& \quad \leq  \frac{c_1^2}{\kappa\lambda_1^{1/2}} \norm{w}_{{H_0}}\norm{w}_{{V_0}}\norm{\theta}_{{V_1}}^2+ \frac{\kappa\lambda_1^{1/2}}{4}\norm{\xi}_{{H_1}}\norm{\xi}_{{V_1}} \notag \\
& \quad \leq \frac \nu 8 \norm{w}_{{V_0}}^2 + \frac{2c_1^4}{\kappa^2\lambda_1\nu} \norm{w}_{{H_0}}^2\norm{\theta}_{{V_1}}^4+ \frac \kappa 8 \norm{\xi}_{{V_1}}^2 + \frac{\kappa\lambda_1}{8} \norm{\xi}_{{H_1}}^2\notag \\
 &\quad \leq \frac \nu 8 \norm{w}_{{V_0}}^2 + \frac{2c_1^4}{\kappa^2\lambda_1\nu} \norm{w}_{{H_1}}^2\norm{\theta}_{{V_1}}^4 + \frac \kappa 4 \norm{\xi}_{{V_1}}^2,
\end{align}
and
\begin{align}\label{2}
\abs{\left(\magenta{B_0}(w,u),w\right)} &\leq \norm{v}_{{V_0}}\norm{w}_{\LpP{4}}^2 \notag \\
&\leq c_1\norm{u}_{{V_0}}\norm{w}_{{H_0}}\norm{w}_{{V_0}}\notag \\
& \leq  \frac \nu 8 \norm{w}_{{V_0}}^2 +\frac {2c_1^2}{\nu} \norm{u}_{{V_0}}^2\norm{w}_{{H_0}}^2.
\end{align}
Thanks to the assumption $\mu c_0^2h^2\leq \nu$ and Young's inequality, we have 
\begin{align}\label{4}
-\mu(I_h(w),w) & = -\mu(I_h(w)-w, w) - \mu \norm{w}_{{H_0}}^2\notag \\
& \leq \mu \norm{I_h(w)-w}_{{H_0}}\norm{w}_{{H_0}} - \mu \norm{w}_{{H_0}}^2\notag \\
&\leq \mu c_0h\norm{w}_{{H_0}}\norm{w}_{{V_0}} - \mu \norm{w}_{{H_0}}^2\notag \\
& \leq \frac{\mu c_0^2h^2}{2}\norm{w}_{{V_0}}^2 - \frac{\mu}{2}\norm{w}_{{H_0}}^2\notag \\
&\leq \frac{\nu}{2}\norm{w}_{{V_0}}^2 -\frac{\mu}{2}\norm{w}_{{H_0}}^2.
\end{align}

Thus, it follows from the estimates \eqref{1}--\eqref{4} that
\begin{align}
& \od{}{t} \left(\norm{w}_{{H_0}}^2 + \norm{\xi}_{{H_1}}^2\right) + \magenta{\frac\nu2}\norm{w}_{{V_0}}^2 + \magenta{\frac\kappa2}\norm{\xi}_{{V_1}}^2 \leq \notag \\
&\qquad \left(\frac{4}{\kappa\lambda_1} + \frac{4c_1^2}{\nu}\norm{u}_{{V_0}}^2 + \frac{4c_1^4}{\kappa^2\lambda_1\nu}\norm{\theta}_{{V_1}}^4 - \mu\right)\norm{w}_{{H_0}}^2.
\end{align}
Using the Poincar\'e inequality \eqref{poincare}, we get
\begin{align}\label{rhs_1}
& \od{}{t} \left(\norm{w}_{{H_0}}^2 + \norm{\xi}_{{H_1}}^2\right) + \magenta{\frac{\lambda_1}2}\min\{\nu,\kappa\}\left(\norm{w}_{{H_0}}^2 + \norm{\xi}_{{H_1}}^2\right) \leq \notag \\
&\qquad\left(\frac{4}{\kappa\lambda_1} + \frac{4c_1^2}{\nu}\norm{u}_{{V_0}}^2 + \frac{4c_1^4}{\kappa^2\lambda_1\nu}\norm{\theta}_{{V_1}}^4 - \mu\right)\norm{w}_{{H_0}}^2.
\end{align}

We denote by
$$\alpha(t) := \mu -\frac{4}{\kappa\lambda_1} - \frac{4c_1^2}{\nu}\norm{u}_{{V_0}}^2 - \frac{4c_1^4}{\kappa^2\lambda_1\nu}\norm{\theta}_{{V_1}}^4.$$
By Proposition \ref{unif_bounds}, there exists $t_0>0$ and such that for all $t\geq t_0$, 
\begin{align}\label{K_1}
& \int_t^{t+1} \norm{u(s)}_{{V_0}}^2\, ds \leq a_3,
\end{align}
and
\begin{align}\label{K_2}
\int_t^{t+1} \norm{\theta(s)}_{{V_1}}^4 \,ds &  \leq \sup_{s\in[t,t+1]}\norm{\theta(s)}_{{V_1}}^2 \int_t^{t+1} \norm{\theta(s)}_{{V_1}}^2 \,ds \leq J_1b_3,
\end{align}
where the positive constants $a_3(\nu,L)$, $b_3(\nu,\kappa,L)$, and $J_1(\nu,\kappa,L)$, are defined in Proposition \ref{unif_bounds}. 
Thus,
\begin{align*}
\limsup_{t\rightarrow\infty} \int_t^{t+1} \alpha(s)\,ds \notag
\geq  \mu -\frac{4}{\kappa\lambda_1} - \frac{4c_1^2a_3}{\nu}- \frac{4c_1^4J_1b_3}{\kappa^2\lambda_1\nu}.
\end{align*}
The assumption \eqref{mu_1}
implies that
\begin{align}\label{cond_1}
\int_t^{t+1} \alpha(s)\,ds \geq \frac\mu 2>0, \quad \text{and} \quad
\int_t^{t+1} \alpha(s)\,ds \leq \frac{3\mu} 2 < \infty.
\end{align}

The inequality \eqref{rhs_1} can be rewritten as
\begin{align*}
\od{}{t} \left(\norm{w}_{\magenta{H_0}}^2 + \norm{\xi}_{{H_1}}^2\right) + \min\{\magenta{\frac{\nu\lambda_1}{2}},\magenta{\frac{\kappa \lambda_1}{2}},\alpha(t)\}\left(\norm{w}_{\magenta{H_0}}^2 + \norm{\xi}_{{H_1}}^2\right)\leq 0.
\end{align*}
Define $\tilde{\alpha}(t): = \min\{\magenta{\frac{\nu\lambda_1}{2}},\magenta{\frac{\kappa \lambda_1}{2}},\alpha(t)\}$, then $\tilde{\alpha}(t)$ satisfies \eqref{cond_1}.
By the uniform Gronwall inequality, Lemma \ref{gen_gron_2}, it follows that
\begin{align*}
\norm{w(t)}_{{H_0}}^2 + \norm{\xi(t)}_{{H_1}}^2 \rightarrow 0,
\end{align*}
at an exponential rate, as $t\rightarrow \infty$.
\end{proof}

Here we have the analogue of Theorem \ref{th_conv_1} but for observable operators that satisfy \eqref{app2_F}.
\begin{theorem}\label{th_conv_2}
Suppose $I_h$ satisfies the approximation property \eqref{app2_F} and $(u(t;x),\theta(t;x))$ \magenta{is} a strong solution, which is contained in the global attractor, of \eqref{Bous_fun}. Let $T>0$, \magenta{$\beta_0>0$ and $\beta_1>0$} be arbitrary, but fixed, such that $\beta_0 \geq J_0$ \magenta{and $\beta_1\geq J_1$}, where $J_0$ \magenta{and $J_1$ are }defined in \eqref{J_0_eps} \magenta{and \eqref{J_1_eps}, respectively}. If $v_0 \in V_0$ and $\eta_0\in {\magenta{H_1}}$ such that
\begin{align}\label{initial_condition}
\magenta{\norm{v_0}_{V_0}^2 \leq\beta_0, \quad \text{and} \quad \norm{\eta_0}_{H_1}^2 \leq \beta_1,}
\end{align}
and $\mu$ is large enough, such that
\begin{align}\label{mu_2}
\mu\geq \frac{\magenta{96c_1^2(\beta_0+\beta_1)}}{\nu}+ \frac{8c_1^2}{\nu\lambda_1}K_3 + \frac{4c_1^2}{\kappa\lambda_1}J_1+ 4\frac{\lambda_1^2+1}{\kappa\lambda_1^2},
\end{align}
and $h>0$ be small enough such that $\mu c_0^2h^2\leq \nu$, where $K_3$ is a positive constant defined in \eqref{K_3}, then \eqref{DA_Bous_fun} has a unique global solution such that
\begin{subequations}\label{strong2}
\begin{align}
v \in C([0,T];V_0)\cap L^2([0,T];\mathcal{D}(A_0)),\\
\eta\in C([0,T];{\magenta{H_1}})\cap L^2([0,T];\magenta{V_1}),
\end{align}
and
\begin{align}
\od{v}{t} \in L^2([0,T];H_0), \qquad \od{\eta}{t}\in L^2([0,T];\magenta{V_1^{'}}).
\end{align}
\end{subequations}

Moreover, $\norm{u(t)-v(t)}_{{V_0}}^2 + \norm{\theta(t)-\eta(t)}_{{H_1}}^2$ $\rightarrow 0$, at an
exponential rate, as $t \rightarrow \infty$.
\end{theorem}

\begin{proof} Define
$w := u-v $ and $\xi := \theta-\eta$.  Since by assumption, $(u, \theta)$ is a solution which is contained in the global attractor of \eqref{Bous_fun}, \magenta{in particular}, it satisfies the global estimates in Proposition \ref{unif_bounds}, then showing the  global existence, in time, of the solution $(w(t), \xi(t))$ is equivalent to showing the global existence, in time, of the solution $(v(t), \eta(t))$ of system \eqref{DA_Bous_fun}. To be concise here, we will show the global existence of the solution $(w(t), \xi(t))$ and show that $\norm{w(t)}_{V_0}^2 + \norm{\xi}_{H_1}^2$ decays exponentially, in time, which will prove the convergence of the approximate solution $(v(t),\eta(t))$ to the exact solution $(u(t), \theta(t))$, exponentially in time.

Taking the difference between system \eqref{Bous_fun} and system \eqref{DA_Bous_fun}, we see that $w$ and $\xi$ satisfy the equations
\begin{subequations}\label{DA_Bous_fun_n_diff}
\begin{align}
&\od{w}{t} +\nu A_0w +B_0(v,w)+ B_0(w,u) = \cP_\sigma(\xi \mathbf{e}_2)- \mu \cP_\sigma I_h(w), \label{w_n}\\
& \od{\xi}{t} - \kappa A_1 \xi +B_1(v,\xi) + B_1(w,\theta)- w\cdot{\bf e}_2= 0, \label{xi_n}\\
&w(0;x) = u_0(x)-v_0(x), \\
&\xi(0;x) = \theta_0(x) - \eta_0(x).
\end{align}
\end{subequations}
Next, we will prove some formal {\it apriori} estimates that are essential in proving the global existence of solutions of system \eqref{DA_Bous_fun_n_diff}. These estimates can be justified rigorously by using the Galerkin method and the Aubin compactness theorem (see e.g. \cite{Constantin_Foias_1988}).

Since \magenta{$\norm{v_0}^2_{V_0} \leq \beta_0$ and $\norm{\eta_0}_{H_1}^2\leq\beta_1$} then by the continuity of \magenta{$\norm{v(t)}_{V_0}^2$ and $\norm{\eta(t)}_{H_1}^2$}, there exists a short time interval $[0,{\tilde{T}})$ such that
\begin{align}\label{V_bound}
\magenta{\norm{v(t)}_{V_0}^2+\norm{\eta(t)}_{H_1}^2 < 12(\beta_0+\beta_1)},
\end{align}
for all $t\in [0,{\tilde{T}})$. Assume $[0, \tilde{T})$ is the maximal interval such that \eqref{V_bound} holds. We will show, by contradiction, that $ \tilde{T} =\infty$. \magenta{Assume that $\tilde{T}<\infty$, then it is clear that 
$$\limsup_{t\rightarrow \tilde{T}^{-}} \left(\norm{v(t)}_{V_0}^2+\norm{\eta(t)}_{H_1}^2\right) = 12(\beta_0+\beta_1), $$
otherwise \eqref{V_bound} will hold beyond $\tilde{T}$.}
Taking the \magenta{$H_0$ and ${H_1}$ inner products of \eqref{w_n} and \eqref{xi_n}, respectively, }with $A_0w$ and $\xi$, respectively, and using the orthogonality property \eqref{orth}, we have
\begin{align*}
\frac 12 \od{}{t} \norm{w }_{{V _0}}^2 & + \nu \norm{A_0w}_{{H_0}}^2 +\left(B_0(v,w ),A_0w\right) +\left(B_0(w ,u ),A_0w\right) \notag \\
&\qquad \qquad \qquad = \int_{\Omega} \xi (A_0w\cdot{\bf e}_2)\, dx - \mu (I_h(w ), A_0w),
\end{align*}
and
\begin{align*}
\frac1 2 \od{}{t} \norm{\xi}_{{H_1}}^2 + \kappa\norm{\xi}_{{V _1}}^2 + \left(B_1(w ,\theta ),\xi\right) &= \int_{\Omega} \xi (w \cdot{\bf e}_2)\, dx. 
\end{align*}
Using integration by parts, H\"older's inequality, Ladyzhenskaya's inequality \eqref{L4_to_H1}, the Poincar\'e inequality \eqref{poincare} and \eqref{V_bound}, we have on the time interval $[0,{\tilde{T}})$:
\begin{align}\label{2b}
\abs{\left(B_0(v,w ),A_0w \right)}&\leq \norm{v}_{{V_0}}\norm{\nabla w }_{\LpP{4}}^2\notag\\
&\leq c_1\norm{v}_{{V_0}}\norm{w }_{{V_0}}\norm{A_0w}_{{H_0}}\notag \\
&\leq \frac{\nu}{8}\norm{A_0w}_{{H_0}}^2 + \frac{2c_1^2}{\nu}\norm{v}_{{V_0}}^2\norm{w }_{{V_0}}^2\notag\\
& \leq \frac{\nu}{8}\norm{A_0w}_{{H_0}}^2 + \frac{\magenta{24c_1^2(\beta_0+\beta_1)}}{\nu}\norm{w}_{{V_0}}^2.
\end{align}
Also, by Ladyzhenskaya's inequality \eqref{L4_to_H1} and the Poincar\'e inequality \eqref{poincare}, we have
\begin{align}\label{3b}
\abs{\left(B_0(w ,u ),A_0w\right)} &\leq \norm{w }_{\LpP{4}}\norm{\nabla u }_{\LpP{4}}\norm{A_0w}_{{H_0}}\notag\\
&\leq c_1 \norm{w}_{H_0}^{1/2}\norm{w}_{V_0}^{1/2}\norm{u }_{V_0}^{1/2}\norm{A_0 u }_{H_0}^{1/2}\norm{A_0w}_{H_0}\notag \\
& \leq \frac{c_1}{\lambda_1^{1/2}}\norm{A_0u}_{{H_0}}\norm{w }_{{V _0}}\norm{A_0w}_{{H_0}}\notag\\
& \leq \frac{\nu}{8}\norm{A_0w}_{{H_0}}^2 + \frac{2c_1^2}{\nu\lambda_1}\norm{A_0u}_{{H_0}}^2\norm{w }_{{V _0}}^2,
\end{align}
and
\begin{align}\label{4b}
\abs{\left(B_1(w ,\theta ),\xi\right)}& \leq \norm{w }_{\LpP{4}}\norm{\xi}_{\LpP{4}}\norm{\theta }_{{V _1}}\notag \\
&\leq c_1\norm{w }_{{H_0}}^{1/2}\norm{w }_{{V _0}}^{1/2}\norm{\xi}_{{H_1}}^{1/2}\norm{\xi}_{{V _1}}^{1/2}\norm{\theta }_{{V_1}}\notag\\
& \leq \frac{c_1}{\lambda_1^{1/2}}\norm{w }_{{V _0}}\norm{\xi}_{{V _1}}\norm{\theta }_{{V _1}}\notag \\
& \leq \frac{\kappa}{4}\norm{\xi}_{{V _1}}^2 + \frac{c_1^2}{\kappa\lambda_1}\norm{\theta }_{{V _1}}^2\norm{w }_{{V _0}}^2.
\end{align}
Integration by \magenta{parts}, the Cauchy-Schwarz and Young inequalities yield
\begin{align}\label{6b}
\abs{\int_{\Omega} \xi (A_0w\cdot{\bf e}_2)\, dx}
&\leq \norm{w }_{{V _0}}\norm{\xi}_{{V _1}}\notag \\
&\leq \frac{\kappa}{4}\norm{\xi}_{{V _1}}^2 + \frac{1}{\kappa}\norm{w }_{{V _0}}^2.
\end{align}
Moreover, by the Cauchy-Schwarz and Young inequalities,
\begin{align}\label{5b}
\abs{\int_{\Omega}\xi (w \cdot{\bf e}_2)\, dx}
&\leq \frac{\kappa}{4}\norm{\xi}_{{V _1}}^2 +  \frac{1}{\kappa\lambda_1^2}\norm{w }_{{V _0}}^2.
\end{align}

Using \eqref{app2_F}, Young's inequality and the assumption $\mu c_0^2h^2 \leq \nu$, we can show that
\begin{align}\label{7b}
-\mu (\mathcal{P}_\sigma I_h(w), A_0w) &= \mu (w - \mathcal{P}_\sigma I_h(w), A_0w) - \mu \norm{w}_{V_0}^2\notag \\
& \leq \frac{\mu^2}{\nu} \norm{w - I_h(w)}_{H_0}^2 + \frac{\nu}{4}\norm{A_0w}_{H_0}^2 - \mu \norm{w}_{V_0}^2 \notag \\
& \leq \frac{\nu}{2} \norm{A_0w}_{H_0}^2 - \frac{\mu}{2} \norm{w}_{V_0}^2.
\end{align}
It follows from the estimates \eqref{2b}--\eqref{7b} and the Poincar\'e inequality \eqref{poincare} that on the time interval $[0,{\tilde{T}})$:
\begin{align}\label{rhs_2}
&\od{}{t}\left(\norm{w }_{{V _0}}^2 + \norm{\xi}_{{H_1}}^2\right) +\frac{\min\left\{\nu,\kappa\right\}}{2}\left(\norm{A_0w }_{{H _0}}^2 + \norm{\xi}_{{V_1}}^2\right)\notag\\
& \quad \leq\left(\frac{\magenta{48c_1^2(\beta_0+\beta_1)}}{\nu}+ \frac{4c_1^2}{\nu\lambda_1}\norm{A_0u}_{{H_0}}^2 + \frac{2c_1^2}{\kappa\lambda_1}\norm{\theta }_{{V _1}}^2 + 2\frac{\lambda_1^2+1}{\kappa\lambda_1^2}-\mu\right)\norm{w }_{{V _0}}^2.
\end{align}
We denote by
\begin{align*}
\alpha(t) &:= \mu - \left(\frac{\magenta{48c_1^2 (\beta_0+\beta_1)}}{\nu}+ \frac{4c_1^2}{\nu\lambda_1}\norm{A_0u}_{{H_0}}^2 + \frac{2c_1^2}{\kappa\lambda_1}\norm{\theta }_{{V _1}}^2 + 2\frac{\lambda_1^2+1}{\kappa\lambda_1^2}\right).
\end{align*}

Since by assumption $(u,\theta)$ is a solution that is contained in the global attractor of \eqref{DA_Bous_fun}, by Proposition \ref{unif_bounds}, one can show that there exist positive constants $K_3= K_3(\nu,\kappa,\lambda_1,L)$ such that for all $t\geq 0$,
\begin{align}\label{K_3}
\norm{A_0u(t)}_{{H_0}}^2 \leq K_3,
\end{align}
moreover, 
\begin{align}
\norm{\theta (t)}_{{V _1}}^2&\leq J_1, 
\end{align}
for all $t>0$. Therefore, assumption \eqref{mu_2}
implies that $\alpha(t)>0$ for all $t\geq 0$.
Consequently, and \magenta{by} virtue of Gronwall's lemma, we have that $\norm{w(t)}_{{V_0}}^2 + \norm{\xi(t)}_{H_1}^2$ is finite and
\begin{align}\label{w_conv_exp}
\norm{w(t)}_{{V_0}}^2 + \norm{\xi(t)}_{H_1}^2 &\leq (\norm{w_0}_{{V_0}}^2 + \norm{\xi_0}_{H_1}^2) e^{-\int_0^t \alpha(s)\,ds}\notag \\
& \leq (\norm{w_0}_{{V_0}}^2 + \norm{\xi_0}_{H_1}^2)e^{-\left(\frac{\magenta{48c_1^2 (\beta_0+\beta_1)}}{\nu}+ \frac{4c_1^2K_3}{\nu\lambda_1} + \frac{2c_1^2J_1}{\kappa\lambda_1} + 2\frac{\lambda_1^2+1}{\kappa\lambda_1^2}\right)t},
\end{align}
for all $t \in [0,{\tilde{T}})$. \magenta{Since $\norm{w_0}_{V_0}^2\leq 4\beta_0$ and $\norm{\xi_0}_{H_1}^2\leq 4\beta_1$, then $\norm{w(t)}_{V_0}^2 +\norm{\xi(t)}_{H_1}^2\leq 4(\beta_0+\beta_1)$}, for all $t\in [0, \tilde{T})$. \magenta{Thus, $\norm{v(t)}_{V_0}^2 + \norm{\eta(t)}_{H_1}^2 \leq 10(\beta_0+\beta_1)$, for all $t\in [0,\tilde{T})$. This, in turn, will lead into a contraction since $$12(\beta_0+\beta_1) = \limsup_{t\rightarrow\tilde{T}^{-}}\left(\norm{v(t)}_{V_0}^2+\norm{\eta(t)}_{H_1}^2\right)\leq 10(\beta_0+\beta_1).$$} This contradicts the assumption that $[0, \tilde{T})$ is the maximal interval such that \eqref{V_bound} holds and proves that $\tilde{T}=\infty$.

This proves that $w(t)$ \magenta{and $\xi(t)$ exist} globally, in time, for all $t\geq 0$ and that \magenta{$\norm{w(t)}_{V_0}^2+\norm{\xi(t)}_{H_1}^2$} decays at an exponential rate. That is, $v(t)$ \magenta{and $\eta(t)$ exist} globally, in time, for all $t\geq 0$ and \magenta{$\norm{v(t)}_{V_0}^2+\norm{\eta(t)}_{H_1}^2 \leq 12(\beta_0+\beta_1)$} for all $t\geq 0$. The proof of the uniqueness of the solution $(v, \theta)$ follows \magenta{similar steps as to those above for proving the} convergence. The estimate \eqref{w_conv_exp} shows that
$$\norm{u(t)-v(t)}_{{V_0}}^2 + \norm{\theta(t)-\eta(t)}_{{H_1}}^2 \rightarrow 0, $$ at an
exponential rate, as $t \rightarrow \infty$.
\end{proof}

\bigskip
\section{discussion and final remarks}

It is typical when implementing data assimilation to choose a relaxation
parameter such as $\mu$ which is effective for the spatial resolution $h$ of the
available data.  The goal of our analysis, however, is to estimate $\mu$ in terms of
physical parameters through rigorous bounds on the solutions in the global attractor.
A sufficiently small value of $h$ is then determined in terms of $\mu$.  Thus, indirectly,
the necessary spatial resolution depends on the physical parameters, which is natural.

We mention that the basis for this reduced assimilation method using velocity alone, namely
Proposition \ref{u_to_theta}, should carry over to the three-dimensional case,
if one could prove the existence of global strong solutions of the 3D B\'enard problem.

It is natural to ask if it is possible to nudge using only temperature measurements. This
would require a temperature to velocity version of Proposition \ref{u_to_theta}, which
remains an open question. It may very well be false, in which case, we would hope to construct a counterexample.

We also plan to carry out numerical work to demonstrate the effectiveness of this approach to data assimilation.
Numerical simulations in \cite{Gesho} (see also \cite{Hayden_Olson_Titi}) have shown that, in the absence
of measurements errors, the continuous data assimilation algorithm \eqref{azouani} performs much better than analytical estimates in \cite{Azouani_Olson_Titi} would suggest. This was also noted for a different data assimilation algorithm in \cite{Olson_Titi_2003} and \cite{Olson_Titi_2008}. It is likely that the data assimilation algorithm studied in this paper will also perform much better than our \magenta{analysis} guarantees.

\bigskip

\section*{Acknowledgements}

This work was completed while the authors were visiting the Institute for Pure and Applied Mathematics (IPAM), which is supported by the National Science Foundation (NSF). The work of A. F. is supported in part by NSF grant  DMS-1418911. The work of M.S.J. is supported in part by NSF grants DMS-1008661, DMS-1109022 and  DMS-1418911. The work of  E.S.T.  is supported in part by the NSF grants  DMS-1009950, DMS-1109640 and DMS-1109645.

\bigskip

\end{document}